\documentclass[12pt]{amsart}
\usepackage[utf8]{inputenc}
\usepackage{amssymb}
\usepackage{graphicx}
\usepackage{graphicx,color}
\usepackage{latexsym}
\usepackage[all]{xy}
\usepackage{mathrsfs}
\usepackage{enumerate}
\usepackage{amssymb}
\usepackage{tikz}
\usepackage{breqn}
\usepackage{bm}
\usepackage{mathptmx}

\makeatletter
\renewcommand{\subsection}{\@startsection
{subsection}{2}{0mm}{\baselineskip}{-0.25cm}
{\normalfont\normalsize\em}}
\makeatother

\def\negei{\mathbf e_i}
\def\negej{\mathbf e_j}

\def\negt{\mathbf t}

\def\negQ{\mathbf Q}
\def\negalpha{\text{\boldmath$\alpha$}}
\def\negsigma{\text{\boldmath$\sigma$}}
\def\neg1{\text{\boldmath$1$}}
\def\negbeta{\text{\boldmath$\beta$}}
\def\neggamma{\text{\boldmath$\gamma$}}
\def\negeta{\text{\boldmath$\eta$}}
\def\negdelta{\text{\boldmath$\delta$}}
\def\neggamma{\text{\boldmath$\gamma$}}

\def\negeta{\text{\boldmath$\eta$}}

\def\neg1{\text{\boldmath$1$}}

\def\hH{\widehat{H}}
\def\cC{\mathcal C}

\def\cL{\mathcal L}
\def\cP{\mathcal P}

\def\cX{\mathcal X}

\def\cM{\mathcal M}

\def\NN{\mathbb{N}}

\def\ZZ{\mathbb{Z}}

\def\FF{\mathbb{F}}
\def\RR{\mathbb{R}}

\def\dis{\displaystyle}

\DeclareMathOperator{\divv}{div}
\DeclareMathOperator{\divvp}{div_\infty}

\newtheorem{theorem}{Theorem}[section]
\newtheorem{proposition}[theorem]{Proposition}
\newtheorem{corollary}[theorem]{Corollary}
\newtheorem{lemma}[theorem]{Lemma}

{\theoremstyle{definition}
\newtheorem{definition}[theorem]{Definition}
\newtheorem{example}[theorem]{Example}

}

{\theoremstyle{remark}
\newtheorem{remark}[theorem]{Remark}
}

\title[Generalized Weierstra\ss~semigroups and Poincaré series]{Generalized Weierstra\ss~semigroups\\ and their Poincaré series}

\author{J.~J.~Moyano-Fern\' andez}
 \address{Universitat Jaume I (UJI),  Campus de Riu Sec, Institut Universitari de Matem\` atiques i Aplicacions de Castell\' o , 12071 Castell\' on
de la Plana, Spain.}
  \email{moyano@uji.es}

\author{W.~Ten\'orio}
 \address{Universidade Federal de Uberlândia (UFU), Faculdade de Matemática, Av. J. N. Ávila 2121, 38408-902, Uberlândia, MG , Brazil}
  \email{dersonwt@yahoo.com.br}

\author{F.~Torres}
  \address{University of Campinas (UNICAMP), Institute of Mathematics, 
Statistics and Computer Science (IMECC), R. 
S\'ergio Buarque de Holanda 651, Cidade 
Universit\'aria \lq\lq Zeferino Vaz", 13083-059, Campinas, SP, Brazil}
  \email{ftorres@ime.unicamp.br}

\begin{document}

\begin{abstract}
We investigate the structure of the generalized Weierstra\ss~semigroups at several points on a curve defined over a finite field. We present
a description of these semigroups that enables us to deduce properties concerned with the arithmetical structure of divisors supported on the specified points and their corresponding Riemann-Roch spaces. This characterization allows us to show that the Poincaré series associated with generalized Weierstra\ss~semigroups carry essential information to describe entirely their respective semigroups.
\end{abstract}

\thanks{The first author was partially supported by the Spanish Government Ministerio de
Econom\'ia, Industria y Competitividad (MINECO), grants MTM2012-36917-C03-03, MTM2015-65764-C3-2-P, MTM2016-81735-REDT and MTM2016-81932-REDT, as well as by Universitat Jaume I, grant P1-1B2015-02. The second author was supported by CNPq, Brazil, under grants 201584/2015-8 and 159852/2014-5, as well as by the IMAC-Institut de Matem\`atiques i Aplicacions de Castell\'o. The third author was supported by CNPq, Brazil, under grant 308326/2014-8.}

\maketitle

\section{Introduction}\label{sec1}

Let $\cX$ be a projective, non-singular, geometrically irreducible algebraic curve over a finite field $\FF$; throughout the paper we will refer to this simply as a \emph{curve}. Let $\FF(\mathcal{X})$ be the field of rational functions of $\cX$ over $\FF$. For $m\geq 1$ pairwise distinct $\FF$-rational points $Q_1,...,Q_m$ on $\cX$, we consider the set
$$
H(\negQ):=\{(\alpha_1,\ldots,\alpha_m)\in \NN_0^m \ : \ \exists \ f\in \FF(\mathcal{X})^\times \ \mbox{with} \ \divvp(f)=\sum_{i=1}^m\alpha_iQ_i\},
$$
where $\divvp(f)$ stands for the pole divisor of the rational function $f$; this is in fact an additive sub-semigroup of $\NN_0^m$ which we will call the \emph{classical Weierstra\ss~semigroup} of $\cX$ at the $m$-tuple $\negQ=(Q_1,...,Q_m)$.

The classical Weierstra\ss~semigroup at one point is a numerical semigroup which goes back to classical works of Riemann, Weierstra\ss~or Hurwitz. An extension to the case of several points was introduced by Arbarello et al.~\cite{A}. The arithmetical properties involved in the special case of two points was extensively investigated by Kim \cite{K} and Homma \cite{H} (see also their joint work \cite{HK}), and references to the general case include Ballico and Kim \cite{BK}, Matthews \cite{Ma1}, Carvalho and Torres \cite{CT}, or the survey by Carvalho and Kato \cite{CK}. Specifically, carrying on the studies of \cite{H,K}, Matthews \cite{Ma1} introduced the notion of \emph{generating set} of $H(\negQ)$ that allows constructing $H(\negQ)$ from the knowledge of a finite number of elements in the semigroup. Since these papers were published, many works have arisen mainly pursuing explicit descriptions of Weierstra\ss~ semigroups of specific curves in order to be applied in the analysis of algebraic-geometric codes.

Motivated by a different interpretation of Weierstra\ss~ numerical semigroups, Delgado \cite{D} provided a generalization to the classical approach of Weierstra\ss~ semigroups at several points on curves over algebraically closed fields; Beelen and Tutas \cite{BT} considered these semigroups in the case of curves over finite fields: writing $v_{Q_i}$ for the valuation of $\FF(\cX)$ associated with $Q_i$, they considered the set of tuples $\rho_\negQ(f):=(-v_{Q_1}(f),\ldots,-v_{Q_m}(f))$, for $f\in \FF(\cX)^{\times}$ with poles only at the points $Q_1,\ldots,Q_m$; this turns out to be the additive sub-semigroup of $\mathbb{Z}^m$
\begin{equation*}\label{gwsem} \hH(\negQ):=\{\rho_\negQ(f)\in \ZZ^m \ : \ f\in R_\negQ\backslash \{0\}\},
\end{equation*}
where $R_\negQ$ denotes the ring of functions of $\cX$ that are regular outside $\{Q_1,\ldots,Q_m\}$, which we call the \emph{generalized Weierstra\ss~ semigroup} of $\cX$ at $\negQ$. 

It is not difficult to see the relation between the \emph{classical} and the \emph{generalized} Weierstra\ss~semigroups, provided that $\#\FF\geq m$ (cf. \cite[Prop. 2]{BT}):
$$
H(\negQ)=\hH(\negQ)\cap \NN_0^m.
$$
An important notion in the study of $\hH(\negQ)$ due to Delgado \cite{D} is that of \emph{maximal elements} (see Def.~\ref{defmaximals}) which in particular allows him to explore symmetry properties of that semigroups. The reader is referred to \cite{Ba1,B,BR,DP,Mo1} for further information.


On the other hand, by using the combinatorial notion of Poincaré series related to a multi-index filtration, the first author \cite{Mo1} introduced the Poincaré series associated with generalized Weierstra\ss~ semigroups for the cases $m=1,2$; there he showed that the knowledge of the series ensures information about the maximal elements of the generalized Weierstra\ss~semigroup.

In this paper we study the \emph{generalized} Weierstra\ss~semigroups at several points and present a generating set in the sense of \cite{Ma1} in terms of the maximal elements. This is related to the characterization of the classical Weierstra\ss~semigroups stated in \cite{Ma1}, but also carries information on the corresponding Riemann-Roch spaces of every divisor supported on the specified points; recall here that any divisor $D=\sum_{i=1}^{m}\mu_iQ_i$ on $\cX$ defines a $\FF$-vector space $\mathcal{L}(D)$---called Riemann-Roch space of $D$---consisting of the rational functions $f$ with poles only at the points with $\mu_i \geq 0$  (and, furthermore, with the pole order of $f$ at $Q_i$ $\leq \mu_i$), and if $\mu_j<0$ having a zero at $Q_j$ or order $\geq \mu_j$.

In particular, we prove that this set of maximal elements---as well as the generating set of $\hH(\negQ)$---is finitely determined. In addition, we extend some results of \cite{Mo1} on the corresponding Poincaré series and conclude that they afford enough information to recover entirely the semigroup, allowing us to interpret these series as a combinatorial invariant of the generalized Weierstra\ss~ semigroups.

Our paper is organized as follows. In Section \ref{sec2} we collect some basic results on the generalized Weierstra\ss~ semigroups and their Poincaré series. Section \ref{sec3} is devoted to the study of a generating set for the generalized Weierstra\ss~ semigroup which is finitely determined; rather than providing such a characterization, we show how these elements are appropriated to describe the Riemann-Roch spaces. In Section \ref{sec4} we investigate the support of the Poincaré series and introduce the semigroup polynomial, which determines the Poincaré series by a functional equation. Finally, we present functional equations for the Poincaré series associated with $\hH(\negQ)$ satisfying a symmetry condition.

\section{Preliminaries}\label{sec2}

Let $\cX$ be a (projective, non-singular, geometrically irreducible, algebraic) curve of genus $g$ defined over a finite field $\FF$ with its function field $\FF(\cX)$. Denote by $\cX (\FF)$ the set of $\FF$-rational points on $\cX$. For $m\geq 2$, let $Q_1,\ldots,Q_m \in \cX(\FF)$ be pairwise distinct. In this section we summarize the main properties of the generalized Weierstra\ss~semigroup $\hH(\negQ)$ defined in the Introduction. We also recall the definition of Poincaré series associated with these algebraic-geometric structures due to \cite{Mo1}.

\subsection{Generalized Weierstra\ss~ semigroups} \label{sec2.1}

Consider the generalized  Weierstra\ss~ semigroup $\hH(\negQ)$ of $\cX$ at $\negQ$. First we want to characterize the elements of $\hH(\negQ)$ in terms of dimensions of Riemann-Roch spaces associated. For this purpose, let us fix some helpful notation: 
\begin{itemize}
\item[$\diamond$] For $\negalpha=(\alpha_1,\ldots,\alpha_m)\in \ZZ^m$, we set $|\negalpha|:=\sum_{i=1}^m \alpha_i$, 
$\cL(\negalpha)$ the Riemann-Roch space associated with the divisor $D(\alpha):=\sum_{i=1}^m \alpha_i Q_i$, and $\ell(\negalpha)$ its $\FF$-dimension.
\item[$\diamond$] Let $I:=\{1,2,\ldots,m\}$. For $\negalpha=(\alpha_1,\ldots,\alpha_m)\in \ZZ^m$ and a non-empty $J\subsetneq I$, we set
$$
\nabla_J(\negalpha):=\{(\beta_1,\ldots,\beta_m)\in \hH(\negQ) \ : \ \beta_j=\alpha_j \ \forall \ i\in J \ \mbox{and} \ \beta_i<\alpha_i \ \forall \ i\in I\backslash J \}
$$
and
$$
\nabla_i(\negalpha):=\nabla_{\{i\}}(\negalpha)
$$
$$
\nabla_i^m(\negalpha):=\{(\beta_1,\ldots,\beta_m)\in \hH(\negQ) \ : \ \beta_i=\alpha_i \ \mbox{and} \ \beta_j\leq\alpha_j \ \text{for} \ j\neq i \}.
$$
\item[$\diamond$] For $\negalpha\in \ZZ^m$, let
$$\nabla(\negalpha):=\bigcup_{i=1}^m \nabla_i(\negalpha).$$
\item[$\diamond$] For any $J\subseteq I$, denote by $\textbf{1}_J$ the $m$-tuple whose the $j$-th coordinate $1$ if $j\in J$ and $0$ otherwise; for instance, $\textbf{1}_I=\textbf{1}$ is the all 1 $m$-tuple, $\textbf{1}_\emptyset=\mathbf{0}$ is the all zero $m$-tuple, and $\negei=\textbf{1}_{\{i\}}$. 
\end{itemize} 

The following result was originally proved by Delgado \cite[p. 629]{D} in the case of algebraically closed fields. For the general case, it can be adapted by using ideas from Carvalho and Torres \cite[Lem. 2.2]{CT}.

\begin{proposition}\label{prop1.1} Let $\negalpha\in \ZZ^m$ and suppose that $\#\FF\geq m$, then
\begin{enumerate}[\rm (1)]
\item  $\negalpha\in \hH(\negQ)$ if and only if $\ell(\negalpha)=\ell(\negalpha-\negei)+1$ for all  $i\in I$;
\item $\nabla_i^m(\negalpha)=\emptyset$ if and only if $\ell(\negalpha)=\ell(\negalpha-\negei)$, for every $i\in I$.
\end{enumerate}  
\end{proposition}
\begin{proof}
(1) If $\negalpha\in \widehat{H}(\negQ)$, then there exists $f\in R_\negQ$ such that $v_{Q_i}(f)=-\alpha_i$ for all $i\in I$, and therefore $f\in \cL(\negalpha)\backslash \cL(\negalpha-\negei)$ for every $i\in I$. Thus, $\ell(\negalpha-\negei)=\ell(\negalpha)-1$ for all $i\in I$. Conversely, from $\ell(\negalpha)=\ell(\negalpha-\negei)+1$ for all $i\in I$, there exists $f_i\in \cL(\negalpha)\backslash \cL(\negalpha-\negei)$ for every $i\in I$. Hence, $v_{Q_i}(f_i)=-\alpha_i$ and $v_{Q_j}(f_i)\geq -\alpha_j$ if $j\neq i$. Consider 
$$
f_i=a_{ij}t_{Q_j}^{v_{Q_j}(f_i)}+\cdots \ \ \  \in \FF(\!(t_{Q_j})\!)
$$
the local expansion of $f_i$ at $Q_j$ as a Laurent series in the indeterminate $t_{Q_j}$. We claim that there exists $(b_1,\ldots,b_m)\in \FF^m$ such that $f=\sum_{i=1}^m b_if_i\in R_\negQ$ and $v_{Q_j}(f)=-\alpha_j$ for all $j\in I$. Indeed, taking $(b_1,\ldots,b_m)\in \FF^m$ outside the union of $m$ hyperplanes, it guarantees that $v_{Q_j}(f)=-\alpha_j$ for every $j$. But as $\#\FF\geq m$, this choice of $(b_1,\ldots,b_m)\in \FF^m$ can always be done.\\
\noindent (2) Note that $\ell(\negalpha)=\ell(\negalpha-\negei)+1$ if and only if there exists $f\in \cL(\negalpha)\backslash \cL(\negalpha-\negei)$, which is equivalent to $\rho_\negQ(f)\in \nabla^m_i(\negalpha)$.
\end{proof}

\begin{remark}
As observed in \cite[p. 214]{CT}, the assumption on the cardinality of $\FF$ is indeed essential for this characterization. 
\end{remark} 

\begin{corollary}  
Let $\negalpha\in \ZZ^m$ and assume that $\#\FF\geq m.$ If $|\negalpha|< 0$ then $\negalpha\not\in \hH(\negQ)$. Furthermore, $\negalpha\in \hH(\negQ)$ whenever $|\negalpha|\geq 2g$.
\end{corollary}

We next recall the notion of maximality due to Delgado in \cite{D}. As we shall see, this concept translates arithmetical properties of Riemann-Roch spaces $\cL(\negalpha)$ into subsets of $\ZZ^m$ related to $\negalpha$.

\begin{definition} \label{defmaximals}
 An element $\negalpha$ of $\hH(\negQ)$ is said \emph{maximal} if $\nabla(\negalpha)=\emptyset$. If moreover $\nabla_J(\negalpha)=\emptyset$ for every non-empty subset $J\subsetneq I$, $\negalpha$ is called an \emph{absolute maximal}. Denote respectively by $\cM(\negQ)$ and $\Gamma(\negQ)$ the sets of maximal  and absolute maximal elements of $\hH(\negQ)$.
\end{definition}

Notice that when $m=2$, every maximal element is also absolute maximal.

\subsection{Poincaré series associated with generalized Weierstra\ss~ semigroups} \label{sec2.2}

Since the ring $R_\negQ$ of functions of $\cX$ that are regular outside $\{Q_1,\ldots,Q_m\}$ can be written as
$$R_\negQ=\bigcup_{\negalpha\in \ZZ^m} \cL(\negalpha),$$ and $\cL(\negalpha)\subseteq \cL(\negbeta)$ whenever $\negalpha\leq \negbeta$ for $\negalpha,\negbeta\in \ZZ^m$, 
we have a multi-index filtration of $R_\negQ$ by Riemann-Roch spaces $\cL(\negalpha)$ which are $\FF$-vector spaces of finite dimension $\ell(\negalpha)$. 
Associated with the filtration $(\cL(\negalpha))_{\negalpha\in \ZZ^m}$, we can define the formal series
\begin{equation*}
L(\negt):=\sum_{\negalpha\in\ZZ^m} d(\negalpha) \cdot\negt^\negalpha,
\end{equation*}
where $\negt^\negalpha$, for $\negalpha=(\alpha_1,\ldots,\alpha_m)\in \ZZ^m$, stands for the monomial $t_1^{\alpha_1}\cdots t_m^{\alpha_m}$, and $$d(\negalpha):=\dim(\cL(\negalpha)/\cL(\negalpha-\textbf{1})).$$

\begin{remark}\label{rmk:formalseries}
By a  $\ZZ$-valued \emph{formal series} (or a \emph{formal distributions}) in the variables $t_1,\ldots,t_m$, we mean a formal expressions of type $S(\negt)=\sum_{\negalpha\in \ZZ^m} s(\negalpha)\negt^\negalpha$ for $s(\negalpha)\in \ZZ$. When $s(\negalpha)=0$ for all but finitely many $\negalpha\in \ZZ^m$, we refer to $S(\negt)$ as a Laurent polynomial. The support of $S(\negt)$ is the set $\{\negalpha\in \ZZ^m \ : \ s(\negalpha)\neq 0\}$. The set of $\ZZ$-valued formal series has a polynomial-like $\ZZ$-module structure in the sense that addition and scalar multiplication are performed like for polynomials. However, the usual multiplication rule can not be used since it does not make sense in general. There, with the usual multiplication rule, the set of $\ZZ$-valued formal series is no longer a ring, but this operation gives us it is a module over the ring of Laurent polynomials over $\ZZ$. For further details on these objects, we refer the reader to \cite{Kac}.
\end{remark}

Observe that when $m=1$, $L(\negt)$ is just a formal power series. It codifies the elements of the Weierstra\ss~ numerical semigroup $\hH(Q_1)=H(Q_1)$, since $\alpha\in \hH(Q_1)$ if and only if $d(\alpha)=1$. We thus have
\begin{equation*}
L(t)=\sum_{\alpha\in \hH(Q_1)} t^\alpha.
\end{equation*}

However, this does not remain true for several points, i.e., there are elements outside $\hH(\negQ)$ appearing in the support of $L(\negt)$; cf. the so-called gaps in \cite{CT}. It thus motivates a more convenient definition of formal series related to $\hH(\negQ)$.

In order to make a precise description of Poincaré series $P(\negt)$ associated with generalized Weierstra\ss~ semigroups, we consider some auxiliary formal series that shall enable us to explore some computational aspects of $P(\negt)$. Let  
$$
Q(\negt):=\prod_{i=1}^m (1-t_i)\cdot L(\negt).
$$
Writing $Q(\negt)=\sum_{\negalpha\in\ZZ^m} q(\negalpha) \cdot\negt^\negalpha$, its coefficients $q(\negalpha)$ are given exactly by
\begin{equation}\label{coef} q(\negalpha)=\sum_{J\in \cP(I)} (-1)^{\#J} d(\negalpha-\textbf{1}_J),\end{equation}
where $\cP(I)$ denotes the power set of $I$.\\

For each $i\in I$, let us also consider the formal series
$$
P_i(\negt):=\sum_{\negalpha\in \ZZ^m} p_i(\negalpha)\cdot \negt^\negalpha
$$
whose coefficients $p_i(\negalpha)$ are
\begin{equation}\label{icoef}
p_i(\negalpha)=(-1)^{m-1}\sum_{J\in \cP(I\backslash\{i\})} (-1)^{\#J} d_i(\negalpha-\textbf{1}+\textbf{1}_J+\negei),
\end{equation}
where $d_i(\negbeta)=\ell(\negbeta)-\ell(\negbeta-\negei)$  for $\negbeta\in \ZZ^m$. 

\begin{lemma} \label{serieaux} Let $\negalpha\in \ZZ^m$ and $i\in I$. The formal series $P_i(\negt)$ does not depend on $i$.
\end{lemma}
\begin{proof}
Let $i,j\in I$ with $i\neq j$. Notice that for any $\negbeta\in \ZZ^m$ we have
\begin{equation} \label{eq:relationij}
d_i(\negbeta)-d_i(\negbeta-\negej)=d_j(\negbeta)-d_j(\negbeta-\negei).
\end{equation}
By \eqref{icoef}, we can deduce that
$$p_i(\negalpha)=\sum_{J\in \cP(I\backslash\{i\})} (-1)^{\#J} d_i(\negalpha-\textbf{1}_J).$$
Hence, as $\cP(I\backslash\{i\})=\cP(I\backslash\{i,j\})\cup \{J\cup \{j\} \ : \ J\in \cP(I\backslash\{i,j\})\}$, we get from \eqref{eq:relationij} that
$$\begin{array}{rcl}
p_i(\negalpha) & = & \displaystyle\sum_{J\in \cP(I\backslash\{i,j\})} (-1)^{\# J} (d_i(\negalpha-\textbf{1}_J)-d_i(\negalpha-\textbf{1}_J-\negej))\\
   & = & \displaystyle\sum_{J\in \cP(I\backslash\{i,j\})} (-1)^{\# J} (d_j(\negalpha-\textbf{1}_J)-d_j(\negalpha-\textbf{1}_J-\negei))\\
   & = & \ p_j(\negalpha).
\end{array}$$
\end{proof}

\begin{definition} The Poincaré series associated with $\hH(\negQ)$ is defined as the multivariate formal series $P_i(\negt)$ in $t_1,\ldots,t_m$. It  will be denoted by $P(\negt)$.
\end{definition}

Observe that, writing $P(\negt)=\sum_{\negalpha\in \ZZ^m} p(\negalpha)\negt^\negalpha$, we have an expression for the coefficients $p(\negalpha)$ of $P(\negt)$, since $P(\negt)=P_i(\negt)$ for $i\in I$. The following result connects the formal series $Q(\negt)$ and $P(\negt)$. It is the analogue of that established in \cite[Prop. 8]{DGN} in our context, and their proofs run very closely with minor adjustments that for completeness we shall indicate below.

\begin{proposition}\label{serieQP} The Poincaré series associated with $\hH(\negQ)$ satisfies the following functional equation
$$
(1-t_1\cdots t_m)\cdot P(\negt)=Q(\negt).
$$
\end{proposition}

\begin{proof} According to the definition of $P(\negt)$, it is sufficient to prove that $p_i(\negalpha)-p_i(\negalpha-\textbf{1})=q(\negalpha)$ for some $i\in I$. First notice that for any reordering $\{i_1,\ldots,i_{m-1}\}$ of $I\backslash\{i\}$ and any $\negbeta\in \ZZ^m$ we can write
\begin{equation*} \label{eq1}
d(\negbeta)=d_{i_1}(\negbeta)+d_{i_2}(\negbeta-\mathbf{e}_{i_1})+\cdots+d_{i_{m-1}}(\negbeta-\mathbf{e}_{i_1}-\cdots-\mathbf{e}_{i_{m-2}})+d_i(\negbeta -\textbf{1}+\mathbf{e}_i)
\end{equation*}
and 
\begin{equation*} \label{eq2}
d(\negbeta+\mathbf{e}_i)=d_i(\negbeta+\mathbf{e}_i)+d_{i_1}(\negbeta)+d_{i_2}(\negbeta-\mathbf{e}_{i_1})+\cdots+d_{i_{m-1}}(\negbeta-\mathbf{e}_{i_1}-\cdots-\mathbf{e}_{i_{m-2}}),
\end{equation*}
from which we deduce
\begin{equation}\label{eq3}
d(\negbeta)-d(\negbeta+\mathbf{e}_i)=d_i(\negbeta-\textbf{1}+\mathbf{e}_i)-d_i(\negbeta+\mathbf{e}_i).
\end{equation}
By \eqref{icoef}, we have
$$
p_i(\negalpha)-p_i(\negalpha-\textbf{1})=(-1)^{m-1}\displaystyle\sum_{J\in \cP(I\backslash\{i\})} (-1)^{\#J} \left(d_i(\negalpha-\textbf{1}+\textbf{1}_J+\mathbf{e}_i)-d_i(\negalpha-\textbf{1}+\textbf{1}_J-\textbf{1}+\mathbf{e}_i) \right)
$$ 
and, by taking $\negbeta=\negalpha-\textbf{1}+\textbf{1}_J$ in \eqref{eq3}, we obtain
$$
p_i(\negalpha)-p_i(\negalpha-\textbf{1})=(-1)^{m}\displaystyle\sum_{J\in \cP(I\backslash\{i\})} (-1)^{\#J} \left(d(\negalpha-\textbf{1}+\textbf{1}_J)-d(\negalpha-\textbf{1}+\textbf{1}_J+\mathbf{e}_i) \right),
$$
which equals to 
$$
(-1)^{m}\displaystyle\sum_{J\in \cP(I\backslash\{i\})} (-1)^{\#J} \left(d(\negalpha-\textbf{1}_{J^c})-d(\negalpha-\textbf{1}_{J^c\backslash\{i\}}) \right),
$$
since $\textbf{1}-\textbf{1}_J=\textbf{1}_{J^c}$. As for any $J\in \cP(I\backslash\{i\})$ we have 
$$
d(\negalpha-\textbf{1}_{J^c})-d(\negalpha-\textbf{1}_{J^c\backslash\{i\}})=(-1)^m \left(d(\negalpha-\textbf{1}_{J})-d(\negalpha-\textbf{1}_{J}-\negei) \right),
$$ 
we thus obtain
$$
p_i(\negalpha)-p_i(\negalpha-\textbf{1})=\displaystyle\sum_{J\in \cP(I\backslash\{i\})} (-1)^{\#J} \left( d(\negalpha-\textbf{1}_J)-d(\negalpha-\textbf{1}_J-\mathbf{e}_i) \right),
$$
which completes the proof by noticing that $\cP(I)=\cP(I\backslash\{i\})\cup\{J\cup\{i\} \ : \ J\in \cP(I\backslash\{i\})\}$ in Eq.~\eqref{coef}.
\end{proof}

In \cite[Sec. 3.2]{Mo1}, Moyano-Fernández specializes the two point case $\negQ=(Q_1,Q_2)$ and obtains that the related Poincaré series is expressed simply as
\begin{equation}\label{Poincareseries2}
P(\negt)=\sum_{\negalpha\in \cM(\negQ)} \negt^\negalpha,
\end{equation}
where $\cM(\negQ)$ is the set of maximal elements of $\hH(\negQ)$. In Section \ref{sec4} we shall investigate the support of the Poincaré series in the general case.

\section{Generalized Weierstra\ss~ semigroups at several points}\label{sec3}

Throughout this section $\cX$ will be a curve over a finite field $\FF$ and $\negQ=(Q_1,...,Q_m)$ will be an $m$-tuple of pairwise distinct $\FF$-rational points on $\cX$. Assume that $\#\FF\geq m\geq 2$. We will provide a generating set of generalized Weierstra\ss~ semigroups at several points which extends the description of Matthews \cite{Ma1} to these objects. We furthermore generalize some results of Beelen and Tutas \cite{BT} to prove the finiteness determination of the generators.

\subsection{Generating sets for generalized Weierstra\ss~ semigroups}\label{sec3.1}

Let $m=2$, and consider the functions $\sigma_i:\ZZ\to \ZZ$ for $i=1,2$ of \cite[Sec. 3]{BT} defined by 
\begin{align*}
\sigma_2(i):=&\min\{t\in \ZZ \ : \ (i,t)\in \hH(\negQ)\}\\
\sigma_1(i):=&\min\{t\in \ZZ \ : \ (t,i)\in \hH(\negQ)\}.
\end{align*}

 It follows from \cite[Prop. 14 (i)]{BT} that
\begin{equation*}\sigma_1(\sigma_2(j))=j=\sigma_2(\sigma_1(j)) \ \text{ for every } \ j\in \ZZ,
\end{equation*}
which yields the equality
\begin{equation}\label{eqmax}
\{(j,\sigma_2(j))\in \hH(\negQ) \ : \ j\in \ZZ\}=\{(\sigma_1(j),j)\in \hH(\negQ) \ : \ j\in \ZZ\};
\end{equation}
these sets in (\ref{eqmax}) coincide exactly with the set of maximal elements $\cM(\negQ)$ of $\hH(\negQ)$ (cf. Definition \ref{defmaximals}). Furthermore, adapting ideas from \cite{K,H}, the maximal elements $\cM(\negQ)$ of $\hH(\negQ)$ provide a generating set for $\hH(\negQ)$ in the sense that
\begin{equation}\label{lub2}\hH(\negQ)=\{(\max\{\beta_1,\beta'_1\},\max\{\beta_2,\beta'_2\})\in \ZZ^2 \ : \ \negbeta,\negbeta'\in \cM(\negQ)\}.
\end{equation}

The next example elucidates the role of the absolute maximal elements in the description of a generalized Weierstra\ss~ semigroup at a pair of points.

\begin{example}\label{ex} This example illustrates the maximal elements of a generalized Weierstra\ss~ semigroup of the Hermitian curve given by the affine equation $x^4=y^3+y$ over $\mathbb{F}_9$ at pair $(Q,P)$, where $Q$ is the point \emph{at infinity} and $P=(0:0:1)$. In Figure \ref{figura1}, the maximal elements are represented by the circles ``$\circ$'', whereas the remaining elements of $\hH(Q,P)$ are represented by the filled circles ``$\bullet$''.
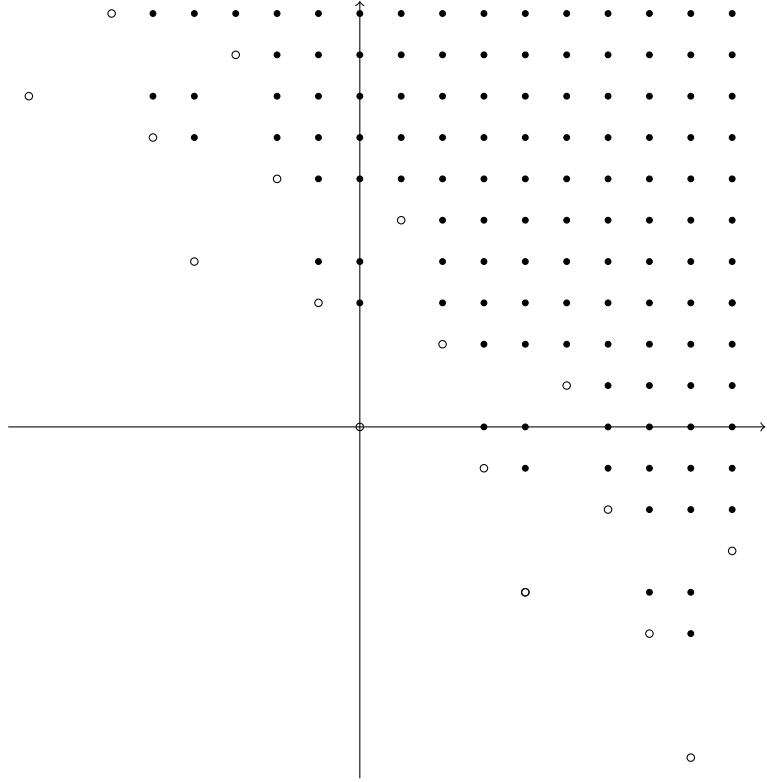
\begin{figure}[!h] 
\centering
\begin{tikzpicture}[scale=0.55] 
\draw [->] (-8.5,0) -- (9.8,0) ;
\draw [->] (0,-8.5) -- (0,10.3) ;

\draw (-8,8) circle [radius=0.09];
\draw (-6,10) circle [radius=0.09];
\draw (-5,7) circle [radius=0.09];
\draw (-4,4) circle [radius=0.09];
\draw (-3,9) circle [radius=0.09];
\draw (-2,6) circle [radius=0.09];
\draw (-1,3) circle [radius=0.09];
\draw (0,0) circle [radius=0.09];
\draw (4,-4) circle [radius=0.09];
\draw (1,5) circle [radius=0.09];
\draw (2,2) circle [radius=0.09];
\draw (3,-1) circle [radius=0.09];
\draw (4,-4) circle [radius=0.09];
\draw (5,1) circle [radius=0.09];
\draw (6,-2) circle [radius=0.09];
\draw (7,-5) circle [radius=0.09];
\draw (8,-8) circle [radius=0.09];
\draw (9,-3) circle [radius=0.09];

\draw[fill] (0,3) circle [radius=0.07];
\draw[fill] (0,4) circle [radius=0.07];
\draw[fill] (0,6) circle [radius=0.07];
\draw[fill] (0,7) circle [radius=0.07];
\draw[fill] (0,8) circle [radius=0.07];
\draw[fill] (0,9) circle [radius=0.07];
\draw[fill] (0,10) circle [radius=0.07];
\draw[fill] (3,0) circle [radius=0.07];
\draw[fill] (4,0) circle [radius=0.07];
\draw[fill] (6,0) circle [radius=0.07];
\draw[fill] (7,0) circle [radius=0.07];
\draw[fill] (8,0) circle [radius=0.07];
\draw[fill] (9,0) circle [radius=0.07];

\draw[fill] (1,6) circle [radius=0.07];
\draw[fill] (1,7) circle [radius=0.07];
\draw[fill] (1,8) circle [radius=0.07];
\draw[fill] (1,9) circle [radius=0.07];
\draw[fill] (1,10) circle [radius=0.07];

\draw[fill] (2,3) circle [radius=0.07];
\draw[fill] (2,4) circle [radius=0.07];
\draw[fill] (2,5) circle [radius=0.07];
\draw[fill] (2,6) circle [radius=0.07];
\draw[fill] (2,7) circle [radius=0.07];
\draw[fill] (2,8) circle [radius=0.07];
\draw[fill] (2,9) circle [radius=0.07];
\draw[fill] (2,10) circle [radius=0.07];

\draw[fill] (3,2) circle [radius=0.07];
\draw[fill] (3,3) circle [radius=0.07];
\draw[fill] (3,4) circle [radius=0.07];
\draw[fill] (3,5) circle [radius=0.07];
\draw[fill] (3,6) circle [radius=0.07];
\draw[fill] (3,7) circle [radius=0.07];
\draw[fill] (3,8) circle [radius=0.07];
\draw[fill] (3,9) circle [radius=0.07];
\draw[fill] (3,10) circle [radius=0.07];

\draw[fill] (4,-1) circle [radius=0.07];
\draw[fill] (4,2) circle [radius=0.07];
\draw[fill] (4,3) circle [radius=0.07];
\draw[fill] (4,4) circle [radius=0.07];
\draw[fill] (4,5) circle [radius=0.07];
\draw[fill] (4,6) circle [radius=0.07];
\draw[fill] (4,7) circle [radius=0.07];
\draw[fill] (4,8) circle [radius=0.07];
\draw[fill] (4,9) circle [radius=0.07];
\draw[fill] (4,10) circle [radius=0.07];

\draw[fill] (5,2) circle [radius=0.07];
\draw[fill] (5,3) circle [radius=0.07];
\draw[fill] (5,4) circle [radius=0.07];
\draw[fill] (5,5) circle [radius=0.07];
\draw[fill] (5,6) circle [radius=0.07];
\draw[fill] (5,7) circle [radius=0.07];
\draw[fill] (5,8) circle [radius=0.07];
\draw[fill] (5,9) circle [radius=0.07];
\draw[fill] (5,10) circle [radius=0.07];

\draw[fill] (6,-1) circle [radius=0.07];
\draw[fill] (6,1) circle [radius=0.07];
\draw[fill] (6,2) circle [radius=0.07];
\draw[fill] (6,3) circle [radius=0.07];
\draw[fill] (6,4) circle [radius=0.07];
\draw[fill] (6,5) circle [radius=0.07];
\draw[fill] (6,6) circle [radius=0.07];
\draw[fill] (6,7) circle [radius=0.07];
\draw[fill] (6,8) circle [radius=0.07];
\draw[fill] (6,9) circle [radius=0.07];
\draw[fill] (6,10) circle [radius=0.07];

\draw[fill] (7,-4) circle [radius=0.07];
\draw[fill] (7,-2) circle [radius=0.07];
\draw[fill] (7,-1) circle [radius=0.07];
\draw[fill] (7,1) circle [radius=0.07];
\draw[fill] (7,2) circle [radius=0.07];
\draw[fill] (7,3) circle [radius=0.07];
\draw[fill] (7,4) circle [radius=0.07];
\draw[fill] (7,5) circle [radius=0.07];
\draw[fill] (7,6) circle [radius=0.07];
\draw[fill] (7,7) circle [radius=0.07];
\draw[fill] (7,8) circle [radius=0.07];
\draw[fill] (7,9) circle [radius=0.07];
\draw[fill] (7,10) circle [radius=0.07];

\draw[fill] (8,-5) circle [radius=0.07];
\draw[fill] (8,-4) circle [radius=0.07];
\draw[fill] (8,-2) circle [radius=0.07];
\draw[fill] (8,-1) circle [radius=0.07];
\draw[fill] (8,1) circle [radius=0.07];
\draw[fill] (8,2) circle [radius=0.07];
\draw[fill] (8,3) circle [radius=0.07];
\draw[fill] (8,4) circle [radius=0.07];
\draw[fill] (8,5) circle [radius=0.07];
\draw[fill] (8,6) circle [radius=0.07];
\draw[fill] (8,7) circle [radius=0.07];
\draw[fill] (8,8) circle [radius=0.07];
\draw[fill] (8,9) circle [radius=0.07];
\draw[fill] (8,10) circle [radius=0.07];

\draw[fill] (9,-2) circle [radius=0.07];
\draw[fill] (9,-1) circle [radius=0.07];
\draw[fill] (9,1) circle [radius=0.07];
\draw[fill] (9,2) circle [radius=0.07];
\draw[fill] (9,3) circle [radius=0.07];
\draw[fill] (9,3) circle [radius=0.07];
\draw[fill] (9,4) circle [radius=0.07];
\draw[fill] (9,5) circle [radius=0.07];
\draw[fill] (9,6) circle [radius=0.07];
\draw[fill] (9,7) circle [radius=0.07];
\draw[fill] (9,8) circle [radius=0.07];
\draw[fill] (9,9) circle [radius=0.07];
\draw[fill] (9,10) circle [radius=0.07];

\draw[fill] (-1,10) circle [radius=0.07];
\draw[fill] (-2,10) circle [radius=0.07];
\draw[fill] (-3,10) circle [radius=0.07];
\draw[fill] (-4,10) circle [radius=0.07];
\draw[fill] (-5,10) circle [radius=0.07];

\draw[fill] (-1,9) circle [radius=0.07];
\draw[fill] (-2,9) circle [radius=0.07];

\draw[fill] (-1,8) circle [radius=0.07];
\draw[fill] (-2,8) circle [radius=0.07];
\draw[fill] (-4,8) circle [radius=0.07];
\draw[fill] (-5,8) circle [radius=0.07];

\draw[fill] (-1,7) circle [radius=0.07];
\draw[fill] (-2,7) circle [radius=0.07];
\draw[fill] (-4,7) circle [radius=0.07];

\draw[fill] (-1,6) circle [radius=0.07];

\draw[fill] (-1,4) circle [radius=0.07];
\end{tikzpicture}
\caption{Maximal elements of $\hH(Q,P)$ in the Example \ref{ex}}
\label{figura1}
\end{figure}
\end{example}

In the remaining of this subsection, we investigate the relationship between maximal elements and generating sets for the generalized Weierstra\ss~ semigroups at several points, taking account the characterization from Eq.~\eqref{lub2} in the case of two points.

The following proposition states equivalences concerning the absolute maximal property.

\begin{proposition} \label{absmax} Let $\negalpha\in \hH(\negQ)$. Then the following statements are equivalent:
\begin{enumerate}[\rm (i)] 
\item $\negalpha\in \Gamma(\negQ)$;
\item $\nabla_i^m(\negalpha)=\{\negalpha\}$ for all $i\in I$;
\item  $\nabla_i^m(\negalpha)=\{\negalpha\}$ for some $i\in I$;
\item $\ell(\negalpha)=\ell(\negalpha-\textbf{1})+1$.
\end{enumerate}
\end{proposition}
\begin{proof} Let us first prove that (i) implies (ii). Since $\negalpha\in \hH(\negQ)$, we get $\{\negalpha\}\subseteq \nabla_i^m(\negalpha)$ for all $i\in I$. If $\negbeta\in \nabla_j^m(\negalpha)$ with $\negbeta\neq \negalpha$ for some $j\in I$, then there exists a subset $J\subsetneq I$ containing $j$ such that $\negbeta\in \nabla_J(\negalpha)$, contradicting the hypothesis.

Since (ii) immediately implies (iii), let us assume (iii). To prove (iv) it is sufficient to show that $\ell(\negalpha-\textbf{1}_{J^c} )=\ell(\negalpha-\textbf{1}_{J^c\cup\{i\}})$ for every subset $J\subsetneq I$ containing $i$, which is equivalent to $\nabla_i^m(\negalpha-\textbf{1}_{J^c})=\emptyset$ by Proposition \ref{prop1.1}(2). But, since 
$$
\nabla_i^m(\negalpha-\textbf{1}_{J^c})\subseteq \nabla_i^m(\negalpha)=\{\negalpha\},
$$ 
the only possibility is $\nabla_i^m(\negalpha-\textbf{1}_{J^c})=\emptyset$ because $\negalpha\not\in \nabla_i^m(\negalpha-\textbf{1}_{J^c})$.

To deduce (i) from (iv), notice that $\ell(\negalpha-\negei)=\ell(\negalpha-\textbf{1})$ for every $i\in I$. If $\nabla_J(\negalpha)\neq \emptyset$ for some non-empty $J\subsetneq I$, since $\nabla_J(\negalpha)\subseteq \nabla_j^m(\negalpha-\textbf{1}_{J^c})$ for any $j\in J$, by Proposition \ref{prop1.1}(2) we get 
$$
\ell(\negalpha-\textbf{1}_{J^c})=\ell(\negalpha-\textbf{1}_{J^c}-\negej)+1.
$$
Since $J\subsetneq I$, it follows that $\ell(\negalpha-\textbf{1}_{J^c})\leq \ell(\negalpha-\negei)$ for $i\in I\backslash J$. As $\ell(\negalpha-\textbf{1}_{J^c}-\negej)\geq \ell(\negalpha-\textbf{1})$, we have
$$
\ell(\negalpha-\negej)\geq \ell(\negalpha-\textbf{1})+1,
$$ which gives us a contradiction. This completes the proof.
\end{proof}

\begin{remark}\label{rmkabsmax}
The main consequence of Proposition \ref{absmax} is that we may regard absolute maximal elements $\negalpha$ as minimal elements in the sets $$\{\negbeta\in \hH(\negQ) \ : \ \beta_i=\alpha_i\} \text{ for } i=1,\ldots,m,$$ 
with respect to the standard Bruhat (partial) order on $\ZZ^m$ given by $$(\alpha_1,\ldots,\alpha_m)\leq (\beta_1,\ldots,\beta_m)\Leftrightarrow \alpha_i\leq \beta_i \ \mbox{for} \ i=1,\ldots,m.$$

This interpretation coincides exactly with the notion used by Matthews in \cite{Ma1} in the formulation of the concept of generating sets for $H(\negQ)$ as well as in the extension of the approach of \cite{K,H} for pairs.
\end{remark} 

\bigskip

Given $\negbeta^1,\ldots,\negbeta^s\in \ZZ^m$, define their least upper bound by
$$
\mbox{lub}(\negbeta^1,\ldots,\negbeta^s):=(\max\{\beta_1^1,\ldots,\beta^s_1\},\ldots,\max\{\beta^1_m,\ldots,\beta^s_m\})\in \ZZ^m.
$$

We next prove a characterization of the generalized Weierstra\ss~ semigroups through least upper bounds of their absolute maximal elements. This description is analogous to that afforded by Matthews in \cite{Ma1} for classical Weierstra\ss~ semigroups at several points.

\begin{theorem} \label{th1} The generalized Weierstra\ss~ semigroup of $\cX$ at $\negQ=(Q_1,\ldots,Q_m)$ can be written as
$$
\hH(\negQ)=\{\mbox{lub}(\negbeta^1,\ldots,\negbeta^m) \ : \ \negbeta^1,\ldots,\negbeta^m\in \Gamma(\negQ)\}.
$$
\end{theorem}
\begin{proof}
To begin with, observe that $\mbox{lub}(\negbeta^1,\negbeta^2)\in \hH(\negQ)$ for $\negbeta_1,\negbeta_2\in \Gamma(\negQ)\subseteq \hH(\negQ)$, since it is always possible to find $(b_1,b_2)\in \FF^2$ outside the union of at most $m$ one-dimensional linear spaces thanks to $\#\FF\geq m$; therefore we may set $h:=b_1 f_1+b_2f_2\in R_\negQ$ such that $\rho_\negQ(h)=\mbox{lub}(\negbeta_1,\negbeta_2)$, where $f_1,f_2\in R_\negQ$ satisfy $\negbeta_1=\rho_\negQ(f_1)$ and $\negbeta_2=\rho_\negQ(f_2)$. This argument may be extended to show that $\mbox{lub}(\negbeta^1,\ldots,\negbeta^m)\in \hH(\negQ)$ for  $\negbeta^1,\ldots,\negbeta^m\in \Gamma(\negQ)$. On the other hand, suppose that $\negalpha\in \hH(\negQ)\backslash \Gamma(\negQ)$. By Proposition \ref{absmax}(iii), it follows that $\nabla_i^m(\negalpha)\supsetneq \{\negalpha\}$  for all $i\in I$, and therefore for each $i\in I$ there exists $\negbeta^i\in \nabla_i^m(\negalpha)$ with $\negbeta^i\neq \negalpha$ such that $\nabla_i^m(\negbeta^i)=\{\negbeta^i\}$. Hence, we can write $\negalpha$ as $\mbox{lub}(\negbeta^1,\ldots,\negbeta^m)$, which concludes the proof.
\end{proof}

As a consequence, the generalized Weierstra\ss~ semigroups are entirely determined by their absolute maximal elements. In what follows, we present the outcomes of this property which seem---at first glance---to justify why this characterization is appropriated to the study of these objects. To be precise, we analyze the relationship between the Riemann-Roch spaces of divisors supported on subsets of $\{Q_1,\ldots,Q_m\}$ and the absolute maximal elements of generalized Weierstra\ss~ semigroups at $\negQ$.\\ 

Given $\negalpha\in \ZZ^m$, let 
\begin{equation*}\label{abssetRR}
\Gamma(\negalpha):=\{\negbeta\in \Gamma(\negQ) \ : \ \negbeta\leq \negalpha\}.
\end{equation*}
For $i\in I$, define on $\Gamma(\negalpha)$ the relation
\begin{equation*}\label{ieqrel}
\negbeta \equiv_i \negbeta' \ \text{if and only if } \ \beta_i=\beta_i'.
\end{equation*}
Notice that $\equiv_i$ is an equivalence relation on $\Gamma(\negalpha)$. Denote by $\Gamma(\negalpha)/\equiv_i$ the set of equivalence classes $[\negbeta]_i$ for $\negbeta\in \Gamma(\negalpha)$. In our next theorem, we formulate a characterization of the dimensions $\ell(\negalpha)$ in terms of absolute maximal elements in $\Gamma(\negalpha)$.

\begin{theorem}\label{th:dimension} 
Let $\negalpha\in \ZZ^m$, then
$$
\ell(\negalpha)=\#(\Gamma(\negalpha)/\equiv_i).
$$
In particular, $\#(\Gamma(\negalpha)/\equiv_i)$ does not depend on $i$.
\end{theorem}
\begin{proof}
We first observe that the conditions imposed on $\negbeta\in \Gamma(\negQ)$ by $\negbeta\leq \negalpha$ and $|\negbeta|\geq 0$ imply that $\Gamma(\negalpha)$ is finite. Note also that the dimension $\ell(\negalpha)$ is precisely the number of proper inclusions in the filtration
$$
\cL(\negalpha)\supseteq \cL(\negalpha-\negei)\supseteq \cL(\negalpha-2\negei)\supseteq \cdots \supseteq \cL(\negalpha-j\negei)\supseteq \cdots \supseteq\{0\}.
$$
By Proposition \ref{prop1.1}(2), $\ell(\negalpha-j\negei)\neq \ell(\negalpha-(j-1)\negei)$ if and only if $\nabla_i^m(\negalpha-j\negei)\neq \emptyset$. It is equivalent to the existence of an absolute maximal element $\negbeta\in \hH(\negQ)$ with $\beta_i=\alpha_i-j$ and $\negbeta\leq \negalpha-j\negei\leq \negalpha$, which, according to Remark \ref{rmkabsmax}, is a minimal element with respect to $\leq$ in the set $\{\neggamma\in \hH(\negQ) \ : \ \gamma_i=\alpha_i-j\}$.
\end{proof}

Write $\Gamma(\negalpha)/\equiv_i \ =\{[\negbeta^1]_i,\ldots,[\negbeta^{\ell(\negalpha)}]_i\}$ for a choice of representative classes, and denote, by convenient abuse of notation, $\rho_\negQ^{-1}(\Gamma(\negalpha)/\equiv_i)=\{f_1,\ldots,f_{\ell(\negalpha)}\}$, where $f_j\in \rho^{-1}_\negQ(\negbeta^j)$ for $j=1,\ldots,\ell(\negalpha).$
Refining the connection established above, we next describe how the absolute maximal elements of $\hH(\negQ)$ carry intrinsic information on $R_\negQ$ as a $\FF$-vector space. 

\begin{corollary}\label{cor:span} 
Let $\negalpha\in\ZZ^m$. The set $\rho_\negQ^{-1}(\Gamma(\negalpha)/\equiv_i)$ is a basis for the Riemann-Roch space $\cL(\negalpha)$. In particular, the ring of functions of $\cX$ that are regular outside $Q_1,\ldots,Q_m$ is spanned by $\rho_\negQ^{-1}(\Gamma(\negQ))$ as an infinite-dimensional $\FF$-vector space.
\end{corollary}
\begin{proof} Observe that $\rho_\negQ^{-1}(\Gamma(\negalpha))\subseteq \cL(\negalpha)$, and so it remains to prove that $\rho_\negQ^{-1}(\Gamma(\negalpha)/\equiv_i)$ is a linearly independent subset of $\cL(\negalpha)$. Indeed, assuming $\ell(\negalpha)>1$, if $\sum_{j=1}^{\ell(\negalpha)}b_j f_j=0$ for $b_j\in \FF$ not all zero, then $\min\{v_{Q_i}(f_1),\ldots,v_{Q_i}(f_{\ell(\negalpha)})\}$ is attained at least two times, which gives a contradiction by the definition of $\equiv_i$. Therefore, $\rho_\negQ^{-1}(\Gamma(\negalpha)/\equiv_i)$ is a basis of $\cL(\negalpha)$. The latter assertion follows by noticing that $R_\negQ=\bigcup_{\negalpha\in \ZZ^m} \cL(\negalpha)$.
\end{proof}

\subsection{Determining generating sets}\label{sec3.2}

Although the generalized Weierstra\ss~ semigroups have the description as in Theorem \ref{th1}, the computation of all absolute maximal elements is not an easy task since the set formed by them is infinite. In this way, we would like to know whether the set $\Gamma(\negQ)$ of absolute maximal elements can be finitely determined. For this purpose, let us start this discussion with the case $m=2$. Here the characterization of elements of $\cM(\negQ)$ as in \eqref{eqmax} is essential since by \cite[Prop. 14 (vii)]{BT}, the function $\sigma_2$  has the periodical property 
\begin{equation} \label{function2}
\sigma_2(j+a)=\sigma_2(j)-a \quad \text{for } j\in \ZZ,
\end{equation}
where $a$ is the smallest positive integer $t$ such that $(t,-t)\in \hH(\negQ)$. Hence, it is possible to determine $\cM(\negQ)$ from the knowledge of both $a$ and the finite set 
\begin{equation}\label{function1}
\{(j,\sigma_2(j)) \ : \ 0\leq j< a\}.
\end{equation}

For the general case we can consider a similar approach: for $i=1,\ldots,m-1$, let $a_i$ be the smallest positive integer $t$ such that $t(Q_i-Q_{i+1})$ is a principal divisor on $\cX$, and denote by $\negeta^i\in \ZZ^m$ the $m$-tuple whose $j$-th coordinate is  
\begin{equation}\label{etas}
\eta^i_j=\left \lbrace \begin{array}{rcl}
a_i ,&  & \mbox{if } j=i;\\
-a_i ,&  & \mbox{if } j=i+1;\\
0 ,&  & \mbox{otherwise.}\\
\end{array}\right.
\end{equation}
Notice that the existence of these $a_i$'s is guaranteed by the finiteness of the divisor class group. Consider the region 
$$\cC:=\{\negalpha\in \ZZ^m \ : \ 0\leq \alpha_i< a_i \ \mbox{for } i=1,\ldots,m-1\}.$$

As noticed in \cite[Sec. 2]{BT}, the functions $f\in R_\negQ^\times$ are such that the support of $\divv(f)$ is a subset of the set $\{Q_1,\ldots,Q_m\}$ and their image under $\rho_\negQ$ form a lattice in the hyperplane $\{\negalpha\in \RR^m \ : \ |\negalpha|=0\}$. Denote by $\Theta(\negQ)$ its sublattice generated by the elements $\negeta^1,\ldots,\negeta^{m-1}$. \\

Similarly to the properties \eqref{function2} and \eqref{function1} satisfied by pairs, we can establish the following general statement on maximal and absolute maximal elements in generalized Weierstra\ss~ semigroups.

\begin{theorem} \label{finitedetermine} 
The maximal elements of $\hH(\negQ)$ are finitely determined by the maximal elements in $\cC$ modulo the lattice $\Theta(\negQ)$; namely,
$$\cM(\negQ)=(\cM(\negQ)\cap \cC)+\Theta(\negQ).$$
Furthermore, this property also holds for absolute maximal elements of $\hH(\negQ)$, that is,
$$\Gamma(\negQ)=(\Gamma(\negQ)\cap \cC)+\Theta(\negQ).$$
\end{theorem}
\begin{proof}
Notice first that if $\negalpha\in \cM(\negQ)\cap \cC$ and $\negeta\in \Theta(\negQ)$, then $\negalpha+\negeta\in \cM(\negQ)$ since otherwise $\neggamma\in \nabla(\negalpha+\negeta)$ implies $\neggamma-\negeta\in \nabla(\negalpha)$, contradicting the maximality of $\negalpha$. On the other hand, given $\negalpha\in \cM(\negQ)$, there exists $\negeta\in \Theta(\negQ)$ such that $\negalpha-\negeta\in \cC$. We claim that $\negalpha-\negeta\in \cM(\negQ)$. Indeed, $\negalpha-\negeta$ belongs to $\hH(\negQ)$ by the semigroup property. If $\neggamma\in \nabla(\negalpha-\negeta)$, then $\neggamma+\negeta \in \nabla(\negalpha)$, contrary to $\negalpha$ being maximal.

Now, suppose that $\negalpha\in \Gamma(\negQ)$. Since $\negalpha$ is maximal, there exists $\negeta\in \Theta(\negQ)$ such that $\negalpha-\negeta\in \cM(\negQ)\cap \cC$. It remains to show that $\negalpha-\negeta\in \Gamma(\negQ)$. On the contrary, from Proposition \ref{absmax}, we would have $\neggamma\in \nabla_i^m(\negalpha-\negeta)$ for some $i\in I$, with $\neggamma\neq \negalpha-\negeta$, which means that $\neggamma+\negeta\in \nabla_i^m(\negalpha)$, contradicting the absolute maximality of $\negalpha$, because $\neggamma+\negeta\neq \negalpha$. On the other hand, given $\negalpha\in \Gamma(\negQ)\cap \cC$, it follows from Proposition \ref{absmax} that $\negalpha+\negeta^j$ is absolute maximal for $j=1,\ldots,m-1$. As each element $\negeta$ in $\Theta(\negQ)$ is an integral linear combination of $\negeta^1,\ldots,\negeta^{m-1}$, we have $\negalpha+\negeta\in \Gamma(\negQ)$. 
\end{proof}

We close this section by exploring the consequences of Theorem \ref{finitedetermine} for the Riemann-Roch spaces associated with divisors whose support is contained in $\{Q_1,\ldots,Q_m\}$. To this end, let us define for $\negalpha, \negalpha'\in \ZZ^m$ the relation
\begin{equation} \label{eqrel}
\negalpha\equiv \negalpha' \ \text{if and only if} \ \negalpha-\negalpha'\in \Theta(\negQ).
\end{equation}
Observe that the foregoing relation is an equivalence relation in $\ZZ^m$ because $\Theta(\negQ)$ is a lattice in $\{\negalpha\in \RR^m \ : \ |\negalpha|=0\}$. Writing $[\negalpha]$ for the equivalence class of $\negalpha$ for $\equiv$, we can state the following property of dimensions in these equivalence classes.

\begin{corollary}\label{dimclass}
Let $\negalpha\in \ZZ^m$, then $\ell(\negalpha')=\ell(\negalpha)$ for any $\negalpha'\in [\negalpha]$.
\end{corollary}
\begin{proof} Note that $\negbeta\in \Gamma(\negQ)$ if and only if $\negbeta+\negeta\in \Gamma(\negQ)$ for any $\negeta\in \Theta(\negQ)$. Hence  $\negbeta\in \Gamma(\negalpha)$ is equivalent to $\negbeta+\negeta\in\Gamma(\negalpha+\negeta)$ for $\negeta\in \Theta(\negQ)$, and thus $\Gamma(\negalpha+\negeta)=\negeta+\Gamma(\negalpha)$. Therefore $$\#(\Gamma(\negalpha+\negeta)/\equiv_i)=\#((\negeta+\Gamma(\negalpha))/\equiv_i)=\#(\Gamma(\negalpha)/\equiv_i),$$
which completes the proof by invoking Theorem \ref{th:dimension}.
\end{proof} 

Since $\Gamma(\negQ)\cap \cC$ is finite, let $\negbeta^1,\ldots,\negbeta^c$ be its elements. For $i=1,\ldots,c$, let $f_i\in R_\negQ$ be a function such that $\rho_\negQ(f_i)=\negbeta^i$. Analogously, let $g_i\in R_\negQ$ be a function such that $\rho_\negQ(g_i)=\negeta^i$ for $i=1,\ldots,m-1$. We can thus state the following concerning such functions.

\begin{corollary}
The ring $R_\negQ$ of functions of $\cX$ that are regular outside $Q_1,\ldots,Q_m$ is spanned as an $\FF$-vector space by the set of functions
$$
\{f_i \cdot g_1^{i_1}\cdots g_{m-1}^{i_{m-1}} \ : \ 0\leq i\leq c \ \ \mbox{and}\ \ i_j\in \ZZ \ \ \mbox{for}\ \ j=1,\ldots,m-1\}.
$$
\end{corollary}
\begin{proof}
According to Corollary \ref{cor:span}, $\rho_\negQ^{-1}(\Gamma(\negQ))$ spans $R_\negQ$ as a $\FF$-vector space. Since $\Gamma(\negQ)=(\Gamma(\negQ)\cap \cC)+\Theta(\negQ)$ by Theorem \ref{finitedetermine}, we have the assertion.
\end{proof}

\section{Poincaré series and their functional equations} \label{sec4}

This section is devoted to the study of the support of Poincaré series associated with generalized Weierstra\ss~ semigroups and their functional equations. In particular, we extend some results of \cite{Mo1}. As in the previous section, let $\cX$ be a curve over a finite field $\FF$ and let $\negQ=(Q_1,\ldots,Q_m)$ be an $m$-tuple of pairwise distinct $\FF$-rational points on $\cX$. We also assume that $\#\FF\geq m\geq 2$.

\subsection{Poincaré series as an invariant} \label{sec4.1}
Eq.~\eqref{icoef} provides an expression for the coefficients of $P(\negt)$. In the proposition to be proved, we use that expression to state a characterization of likely elements in the support of $P(\negt)$, extending somehow the formula \eqref{Poincareseries2} to the case of several points. This result is the version of \cite[Prop. 3.8]{Mo2} for our Poincaré series associated with generalized Weierstra\ss~ semigroups; we shall point its proof out for the sake of clarity.

\begin{proposition}\label{prop3} 
Let $P(\negt)=\sum_{\negalpha\in \ZZ^\ell} p(\negalpha)\negt^\negalpha$ be the corresponding Poincaré series of $\hH(\negQ)$, then the following statements hold:
\begin{enumerate}[\rm(a)]
\item if $\negalpha\not\in \hH(\negQ)$ then $p(\negalpha)=0$;
\item if $\negalpha\in \hH(\negQ)\backslash \cM(\negQ)$ then $p(\negalpha)=0$;
\item if $\negalpha\in \Gamma(\negQ)$ then $p(\negalpha)=1$.
\end{enumerate}
\end{proposition}

\begin{proof} 
By Proposition \ref{prop1.1} (2) we have $d_i(\negbeta)=0$ if and only if $\nabla_i^m(\negbeta)=\emptyset$; thus for any $i, j\in I$ with $i\neq j$ and $\negbeta\in \ZZ^m$, we get
$$
d_i(\negbeta+\mathbf{e}_i+\mathbf{e}_j)\geq d_i(\negbeta+\mathbf{e}_i),
$$
since $\nabla_i^m(\negbeta+\mathbf{e}_i+\mathbf{e}_j)\supseteq \nabla_i^m(\negbeta+\mathbf{e}_i)$. Hence, for any $i\in I$ and any reordering $\{i_1,\ldots,i_{m-1}\}$ of $I\backslash\{i\}$, we have
\begin{equation}\label{eqdim}
0\leq d_i(\negalpha-\textbf{1}+\mathbf{e}_i)\leq d_i(\negalpha-\textbf{1}+\mathbf{e}_i+\mathbf{e}_{i_1})\leq \cdots\leq d_i(\negalpha)\leq 1.
\end{equation}
To prove (a), we observe that $d_i(\negalpha)=0$ for some $i\in I$ by Proposition \ref{prop1.1}. Then, by \eqref{eqdim} we deduce that each sum $d_i(\negalpha-\textbf{1}+\textbf{1}_J+\negei)$ of the coefficients $p_i(\negalpha)$ in \eqref{icoef}
vanishes and therefore $p_i(\negalpha)=0$, which gives (a) by Proposition \ref{serieQP}. Now, assuming $\negalpha\in \hH(\negQ)\backslash \cM(\negQ)$, then $\nabla_i(\negalpha)\neq \emptyset$ for some $i\in I$, which implies $d_i(\negalpha-\textbf{1}+\mathbf{e}_i)=1$ by the equality $\nabla_i(\negalpha)=\nabla_i^m(\negalpha-\textbf{1}+\mathbf{e}_i)$ together with Proposition \ref{prop1.1} (2). Therefore, all other following terms in the inequalities \eqref{eqdim} are equal to 1 and so
\begin{dgroup} 
\begin{dmath}\label{auxP1}
p_i(\negalpha)  =  {(-1)^{m-1}\displaystyle\sum_{j=0}^{m-1} (-1)^j\displaystyle\sum_{J\in \cP(I\backslash\{i\}) \atop \#J=j} d_i(\negalpha-\textbf{1}+\textbf{1}_J+\negei)}\end{dmath}
\begin{dmath}\label{auxP2}
		= {(-1)^{m-1}\displaystyle\sum_{j=0}^{m-1} (-1)^j {m-1 \choose j},}
\end{dmath}
\end{dgroup}
which yields $p_i(\negalpha)=0$, since the sum in \eqref{auxP2} is 0. For $\negalpha\in \Gamma(\negQ)$, Proposition \ref{prop1.1} implies that $d_i(\negalpha)=1$ since $\negalpha\in \hH(\negQ)$, and $d_i(\negalpha-\textbf{1}+\textbf{1}_J+\negei)=0$ for all $J\subsetneq I\backslash\{i\}$ because $d(\negalpha)=1$ by Proposition \ref{absmax}. The proof concludes by observing that from \eqref{auxP1} we obtain $p_i(\negalpha)=(-1)^{m-1}(-1)^{m-1}=1$.  
\end{proof}

Therefore, from the items (a) and (b) in Proposition \ref{prop3}, we can rewrite the Poincaré series $P(\negt)$ associated with $\hH(\negQ)$ as 
\begin{equation}\label{maxeq}
P(\negt)=\sum_{\negalpha\in \cM(\negQ)}p(\negalpha)\negt^\negalpha.
\end{equation}
Notice that in this form, $P(\negt)$ agrees exactly with that afforded by \eqref{Poincareseries2} for the two point case. Furthermore, by item (c), the absolute maximal elements of $\hH(\negQ)$ do appear in the support of $P(\negt)$.

We next show a general expression for the Poincaré series which allows us to describe them from a certain Laurent polynomial and the generators of $\Theta(\negQ)$.

\begin{theorem} \label{serfindet}
Let $P^*(\negt)$ be the multivariate Laurent polynomial 
\begin{equation} \label{laurentpolyn} P^*(\negt):=\sum_{\negalpha\in \cM(\negQ)\cap \cC} p(\negalpha)\negt^\negalpha.
\end{equation} 
Then the Poincaré series $P(\negt)$ associated with $\hH(\negQ)$ satisfies 
\begin{equation*} P(\negt)=\left(\sum_{\negeta\in \Theta(\negQ)}\negt^\negeta\right)\cdot P^*(\negt).\end{equation*}
\end{theorem}

\begin{proof}
After writing $P(\negt)$ as in \eqref{maxeq}, the proof follows by Theorem \ref{finitedetermine} provided that we show the equality $p(\negalpha)=p(\negalpha+\negeta)$ for any $\negalpha\in \cM(\negQ)\cap \cC$ and $\negeta\in \Theta(\negQ)$. Since the coefficient $p(\negalpha)$ is as in \eqref{icoef} for any $i\in I$, it is enough to prove that $d_i(\negalpha-\textbf{1}+\textbf{1}_J+\negei)=d_i(\negalpha+\negeta-\textbf{1}+\textbf{1}_J+\negei)$ for any $J\in \cP(I\backslash\{i\})$, which is equivalent to show that $\nabla_i^m(\negalpha)=\emptyset$ if and only if $\nabla_i^m(\negalpha+\negeta)=\emptyset$. But this holds by the fact that $\negeta,-\negeta\in \hH(\negQ)$.
\end{proof}

The Laurent polynomial $P^*(\negt)$ in \eqref{laurentpolyn} satisfying the functional equation of Theorem \ref{serfindet} is called the \emph{semigroup polynomial} associated with $\hH(\negQ)$.\\

We now conclude the current section by proving a result which summarizes the main properties of Poincaré series associated with a Weierstra\ss~ semigroup.

\begin{theorem}\label{th:main} The Poincaré series associated with a generalized Weierstra\ss~ semigroup at several rational points of $\cX$ determines completely the whole semigroup. Furthermore, it is determined by its semigroup polynomial.
\end{theorem}
\begin{proof} 
According to Theorem \ref{th1}, the absolute maximal elements in $\hH(\negQ)$ determine entirely the semigroup via least upper bounds. Since Proposition \ref{prop3} (c) asserts that the coefficients of absolute maximal elements in the Poincaré series are non-zero, the Poincaré series carries sufficient information to determine the whole semigroup. Moreover, Proposition \ref{serfindet} states that the semigroup polynomial $P^*(\negt)$ determines finitely the Poincaré series $P(\negt)$ through the functional equation of Theorem \ref{serfindet}. This concludes the proof.
\end{proof}

\subsection{Symmetry and functional equations} \label{sec4.2}
In \cite[Sec. 4]{Mo1}, Moyano-Fernández described a functional equation for the Poincaré series of generalized Weierstra\ss~ semigroups satisfying a condition of symmetry. Here we introduce an equivalent notion of symmetry to that in \cite{Mo1} which will enable us to extend the results there.

\begin{definition}\label{defsymm}
 We say that $\hH(\negQ)$ is \emph{symmetric} if there exists an element $\neggamma\not\in \hH(\negQ)$ with $|\neggamma|=2g-1$, where $g$ is the genus of $\cX$.
\end{definition}

\begin{remark}\label{rmkcanonical}
\begin{enumerate}[\rm (1)]
\item Observe that this notion of symmetry accords and extends precisely that in the case $m=1$, which says that a Weierstra\ss~ semigroup $\hH(Q)=H(Q)$ is symmetric if $2g-1\not\in \hH(Q)$, where $g$ is genus of the curve. However, Definition \ref{defsymm} does not make sense for subsemigroups of $\ZZ^m$ that are not generalized Weierstra\ss~ semigroups.
\item Notice also that $\ell(\neggamma)=g$, since $|\neggamma|=2g-1$. As $\neggamma\not\in \hH(\negQ)$, it follows by Proposition \ref{prop1.1}(1) that $\ell(\neggamma-\negei)=\ell(\neggamma)$ for some $i\in I$. Since $|\neggamma-\negei|=2g-2$ and $\ell(\neggamma-\negei)=g$, the divisor $D(\neggamma-\negei)$ is a canonical divisor on $\cX$. Therefore, the symmetry property of $\hH(\negQ)$ implies the existence of a canonical divisor supported on a subset of $\{Q_1,\ldots,Q_m\}$. This actually yields an equivalence to other properties satisfied by symmetric generalized Weierstra\ss~ semigroups, as we will see in Proposition \ref{symmetry}.
\end{enumerate} 
\end{remark}

The next lemma generalizes the result stated in \cite[p. 629]{D} to non-algebraically closed fields. 
Although their proof follows similar lines, some adaptations will be necessary.

\begin{lemma}\label{lem1.1}  
Let $\negalpha=(\alpha_1,\ldots,\alpha_m)\in \ZZ^m$.  Then $\nabla(\negalpha)=\emptyset$ if and only if there exists a canonical divisor $K$ on $\cX$ with $v_{Q_i}(K)=\alpha_i-1$ for $i=1,\ldots,m$, and $K\geq D(\negalpha-\textbf{1})$.
\end{lemma}
\begin{proof} 
By Proposition \ref{prop1.1} (2) and the Riemann-Roch theorem, we have the following equivalences:
$$
\begin{array}{rcl}
\nabla(\negalpha)=\emptyset & \Longleftrightarrow & \nabla^m_j(\negalpha-\textbf{1}+\negej)=\emptyset \text{ for all }j\in I \\
&\Longleftrightarrow & \ell(\negalpha-\textbf{1}+\negej)=\ell(\negalpha-\textbf{1}) \text{ for all }j\in I\\
&\Longleftrightarrow & \ell(K'-\negalpha+\textbf{1}-\negej)+1=\ell(K'-\negalpha+\textbf{1}) \text{ for all }j\in I\\
& \Longleftrightarrow & \exists \ f_j\in \cL(K'-\negalpha+\textbf{1})\backslash \cL(K'-\negalpha+\textbf{1}-\negej) \text{ for all }j\in I,
\end{array}
$$ 
where $K'$ is a canonical divisor on $\cX$. Therefore, $\nabla(\negalpha)=\emptyset$ is equivalent to the existence of $f_j\in \FF(\cX)^\times$ for $j=1,\ldots,m$, satisfying
$$
v_{Q_j}(f_j)=\alpha_j-v_{Q_j}(K')-1 \ \ \mbox{and} \ \ v_{Q_i}(f_j)\geq \alpha_i-v_{Q_i}(K')-1 \ \mbox{if} \ i\neq j.
$$
By the assumption $\#\FF\geq m$, we can proceed exactly as in the proof of Proposition \ref{prop1.1} to choose an appropriate element $(b_1,\ldots,b_m)\in \FF^m$ in order to get $f=\sum_{j=1}^m b_jf_j\in R_\negQ$ with $v_{Q_i}(f)=\alpha_i-v_{Q_i}(K')-1$ for all $i\in I$, since for any $Q\neq Q_i$ on $\cX$ we obtain $v_Q(f)\geq \min_{j=1}^m \{v_Q(f_j)\}\geq -v_Q(K')$. Consequently, the canonical divisor $K=\divv(f)+K'$ satisfies $v_{Q_i}(K)=\alpha_i-1$ and $K-D(\negalpha-\textbf{1})\geq 0$, which is the desired conclusion.
\end{proof}

We now present an additional equivalence to the statements established by Delgado in \cite[p. 630]{D} and revisited by Moyano-Fernández in \cite[Th. 3]{Mo1}.

\begin{proposition}\label{symmetry} The following conditions are equivalent:
\begin{enumerate}[\rm(1)]
\item $\hH(\negQ)$ is symmetric;
\item There exists a canonical 
divisor $K$ on $\cX$ such that $\mbox{Supp}(K)\subseteq\{Q_1,\dots,Q_m\}$;
\item There exists $\negsigma\in \hH(\negQ)$ with $|\negsigma|=2g-2+m$ satisfying the property:
$$\mbox{If} \ \negbeta\in \ZZ^m, \ \mbox{then} \ \negbeta\in \hH(\negQ) \ \mbox{if and only if} \ \nabla(\negsigma-\negbeta)=\emptyset.$$ 
\end{enumerate}
\end{proposition}

\begin{proof}
$(1)\Rightarrow (2)$ This has been already considered in Remark \ref{rmkcanonical}(2).

\noindent$(2)\Rightarrow (3)$ Let $K$ be a canonical divisor on $\cX$ such that $K=D(\negdelta)$ for $\negdelta\in \ZZ^m$ with $|\negdelta|=2g-2$. Taking $\negsigma=\negdelta+\textbf{1}$, we have $\negsigma\in \hH(\negQ)$ because $|\negsigma|=2g-2+m$ with $m\geq 2$. If $\negbeta\in \hH(\negQ)$, then $\nabla(\negsigma-\negbeta)=\emptyset$ since otherwise, one would have $\nabla(\negsigma)\neq \emptyset$.  Conversely, for any $\negbeta\in\ZZ^m$, assume $\nabla(\negsigma-\negbeta)=\emptyset$. Then, by Lemma \ref{lem1.1}, there exists a canonical divisor $K'$ on $\cX$ such that $v_{Q_i}(K')=\sigma_i-\beta_i-1$ for all $i\in I$ and $K'\geq D(\negsigma-\negbeta-\textbf{1})$. So there exists $f\in \FF(\cX)^\times$ such that $K'=K+\divv(f)$. We claim that $f\in R_\negQ$: since $v_Q(f)=v_Q(K')-v_Q(K)=v_Q(K')$ and $K'\geq D(\negsigma-\negbeta-\textbf{1})$, then $v_Q(f)\geq 0$ if $Q\neq Q_i$ and $-v_{Q_i}(f)=(\sigma_i-1)-(\sigma_i-\beta_i-1)=\beta_i$ for every $i\in I$. Therefore, $\negbeta\in \hH(\negQ)$.

\noindent $(3)\Rightarrow (1)$ Note that $\negsigma-\textbf{1}+\negei\not\in  \hH(\negQ)$ with $|\negsigma-\textbf{1}+\negei|=2g-1$ for all $i\in I$, since otherwise $\negsigma-\textbf{1}+\negei \in \nabla_i(\negsigma)$, which cannot happen because $\negsigma\in \cM(\negQ)$.
\end{proof}

\begin{remark}\label{rmkfim} 
\begin{enumerate}
\item From the proof above, if an element $\negsigma\in \widehat{H}(\negQ)$ with $|\negsigma|=2g-2+m$, then $D(\negsigma-\textbf{1})$ is a canonical divisor on $\cX$. Hence, there exists a canonical divisor $K$ on $\cX$ whose support is exactly $\{Q_1,\ldots,Q_m\}$ if and only if there exists an element $\negsigma\in \widehat{H}(\negQ)$ with $|\negsigma|=2g-2+m$ and $\sigma_i\neq 1$ for $i=1,\ldots,m$.
\item If $\hH(\negQ)$ is symmetric, then for each $\negalpha\in \cM(\negQ)$ there exists an unique $\negbeta\in \cM(\negQ)$ such that $\negalpha+\negbeta=\negsigma$. Indeed $\negbeta:=\negsigma-\negalpha \in \hH(\negQ)$ since $\nabla(\negsigma-\negbeta)=\nabla(\negalpha)=\emptyset$. Furthermore, $\nabla(\negbeta)=\emptyset$ because $\negalpha\in \hH(\negQ)$. Therefore, $\negbeta=\negsigma-\negalpha\in \cM(\negQ)$. As a consequence, $\negsigma\in \cM(\negQ)$ since $\textbf{0}\in \cM(\negQ)$.
\end{enumerate}
\end{remark}

The following technical lemma provides a relation that will be used in Theorem \ref{th:fuceq} to deduce a functional equation for Poincaré series associated with symmetric generalized Weierstra\ss~ semigroups. 

\begin{lemma} \label{tlemma}
 Let $\negalpha\in \ZZ^m$ and $i\in I$. If $\hH(\negQ)$ is symmetric, then 
$$
d_i(\negalpha)+d_i(\negsigma-\negalpha-\textbf{1}+\negei)=1,
$$
where $\negsigma$ is a maximal element of $\hH(\negQ)$ as in Proposition \ref{symmetry} (3).
\end{lemma}
\begin{proof} Since by Remark \ref{rmkfim} (1) the divisor $D(\negsigma-\textbf{1})$ is canonical on $\cX$, the Riemann-Roch theorem gives us the equalities
\begin{align*}
\ell(\negalpha)&=|\negalpha|+1-g+\ell(\negsigma-\negalpha-\textbf{1})\\
\ell(\negalpha-\negei)&=|\negalpha|-g+\ell(\negsigma-\negalpha-\textbf{1}+\negei),
\end{align*}
which completes the proof by subtracting the latter from the first.
\end{proof}

In \cite[Prop. 6]{Mo1}, the Poincaré series associated with symmetric generalized Weierstra\ss~ semigroups at two points are shown to satisfy the functional equation
\begin{equation} \label{feq}
P(\negt)=\negt^{\negsigma} P(\negt^{-1}),
\end{equation}
where $\negsigma$ is as in Proposition \ref{symmetry} (3).  We can now formulate a generalization of the functional equations \eqref{feq} to multivariable Poincaré series of generalized Weierstra\ss~ semigroups at several points.

\begin{theorem} \label{th:fuceq}
If $\hH(\negQ)$ is symmetric, then its corresponding Poincaré series satisfies 
$$P(\negt)=(-1)^m\negt^\negsigma P(\negt^{-\mathbf{1}}),$$
where $\negsigma$ is a maximal element of $\hH(\negQ)$ as in Proposition \ref{symmetry} (3).
\end{theorem}
\begin{proof} 
Let $\negalpha\in \ZZ^m$. Notice that $d_i(\negalpha-\textbf{1}+\textbf{1}_J+\negei)=d_i(\negalpha-\textbf{1}_{J^c})$ for every $J\in \cP(I\backslash\{i\})$. Since $(-1)^{\#J^c}=(-1)^{m-1}(-1)^{\#J}$, we can write
$$
p(\negalpha)=p_i(\negalpha)=(-1)^{m-1}\sum_{J\in \cP(I\backslash\{i\})}(-1)^{\#J}d_i(\negalpha-\textbf{1}+\textbf{1}_J+\negei)=\sum_{J\in \cP(I\backslash\{i\})}(-1)^{\#J} d_i(\negalpha-\textbf{1}_J)
$$
and 
$$
(-1)^{m-1}p(\negsigma-\negalpha)=(-1)^{m-1}p_i(\negsigma-\negalpha)=\sum_{J\in \cP(I\backslash\{i\})}(-1)^{\#J} d_i(\negsigma-\negalpha-\textbf{1}+\textbf{1}_J+\negei).
$$
Lemma \ref{tlemma} yields 
\begin{align*}
p(\negalpha)+(-1)^{m-1}p(\negsigma-\negalpha)&= \sum_{J\in \cP(I\backslash\{i\})}(-1)^{\#J}[d_i(\negalpha-\textbf{1}_J)+d_i(\negsigma-\negalpha-\textbf{1}+\textbf{1}_J+\negei)]\\
&=0.
\end{align*}
Hence, using the equality above together with Eq.~\eqref{maxeq} for $P(\negt)$, Remark \ref{rmkfim} (2) implies that 
\begin{equation*}
\begin{array}{rcl}
P(\negt) & = & \dis\sum_{\negalpha\in \cM(\negQ)}p(\negalpha)\negt^\negalpha\\
		 & = & (-1)^m \dis\sum_{\negalpha\in \cM(\negQ)}p(\negsigma-\negalpha)\negt^\negalpha\\
		 & = & (-1)^{m}\dis\sum_{\negbeta\in \cM(\negQ)}p(\negbeta)\negt^{\negsigma-\negbeta}\\
		 & = & (-1)^m \negt^\negsigma \dis\sum_{\negbeta\in \cM(\negQ)}p(\negbeta)\negt^{-\negbeta}\\
		 & = & (-1)^m \negt^\negsigma P(\negt^{-1}).
\end{array}		
\end{equation*}
\end{proof}
In particular, we can also derive an functional equation for $Q(\negt)$ as follows.

\begin{corollary}
If $\hH(\negQ)$ is symmetric, the functional equation
$$
Q(\negt)=(-1)^{m-1}\negt^\negsigma Q(\negt^{-\mathbf{1}})
$$
holds, where $\negsigma$ is a maximal element of $\hH(\negQ)$ as in Proposition \ref{symmetry} (3).
\end{corollary}
\begin{proof}
Let $\negalpha\in \ZZ^m$. By Proposition \ref{serieQP} and the relation $p(\negalpha)=(-1)^{m}p(\negsigma-\negalpha)$ in the proof of Theorem \ref{th:fuceq}, we get
\begin{equation*}
\begin{array}{rcl}
q(\negalpha) & = & p(\negalpha)-p(\negalpha-\textbf{1}) \\
 			 & = & (-1)^m( p(\negsigma-\negalpha)-p(\negsigma-\negalpha+\textbf{1})) \\
 			 & = & (-1)^{m-1} (p(\negsigma-\negalpha+\textbf{1})-p(\negsigma-\negalpha)) \\
 			 & = & (-1)^{m-1} q(\negsigma-\negalpha+\textbf{1}),
 			 \end{array}
\end{equation*}
which proves the assertion.
\end{proof}

\section*{Acknowledgements}
This paper was partially written during a visit of the second and third authors to the University of Valladolid, and a visit of the second author to the University Jaume I of Castell\'on. Both wish to thank the institutions for the hospitality and  support. In particular, the authors wish to thank A. Campillo for many stimulating discussions on Poincaré series.

\end{document}